\documentclass[a4paper,reqno,12pt]{amsart}

\usepackage[T1]{fontenc}
\usepackage{amsthm} 
\usepackage{amsmath, amssymb, epsfig}
\usepackage{url}
\usepackage[mathscr]{euscript}
\usepackage{color}
\usepackage{harpoon}
\usepackage{tikz}
\usepackage{pgfplots}
\pgfplotsset{compat=1.18}
\usepackage{float}
\usetikzlibrary{shapes.geometric, calc}
\usepackage{newtxtext}
\usepackage{newtxmath}
\usepackage{hyperref}

\usepackage{fullpage} 
\usepackage{setspace}
\usepackage{adjustbox}

\usepackage{mathtools}
\mathtoolsset{showonlyrefs}

\def\today{\ifcase\month\or
  January\or February\or March\or April\or May\or June\or
  July\or August\or September\or October\or November\or December\fi
  \space\number\day, \number\year}

\newtheorem{theorem}{Theorem}
\newtheorem{conjecture}{Conjecture}
\newtheorem{lemma}[theorem]{Lemma}
\newtheorem{proposition}[theorem]{Proposition}\newtheorem{corollary}[theorem]{Corollary}

\newcommand{\I}{\mathscr{I}}
\newcommand{\J}{\mathscr{J}}

\newcommand{\LL}{\mathfrak{L}}

\newcommand{\z}{\mathbb{Z}}

\renewcommand{\r}{\mathbb{R}}

\newcommand{\bo}{\boldsymbol}

\newcommand{\la}{\lambda}

\newcommand{\al}{\alpha}
\newcommand{\be}{\beta}
\newcommand{\ep}{\varepsilon}

\newcommand{\si}{\sigma}
\newcommand{\ka}{\kappa}

\newcommand{\1}{\bo{1}}
\renewcommand{\d}{\mathrm{d}}
\newcommand{\del}{\delta}

\title{Sharp Lower Bounds For Sumsets in Hypercubes}
\author[Gonçalves and Radchenko]{Felipe Gonçalves and Danylo Radchenko}
\date{\today}
\subjclass[2010]{}
\address{IMPA - Instituto de Matemática Pura e Aplicada, Rio de Janeiro, 22460-320, Brazil.}
\email{goncalves@impa.br}
\address{Institut des Hautes \'Etudes Scientifiques, CNRS, Laboratoire Alexander Grothendieck,
35 route de Chartres, Bures-sur-Yvette 91440, France}
\email{danradchenko@gmail.com}
\allowdisplaybreaks
\usetikzlibrary{patterns}

\begin{document}

\begin{abstract}
We prove a sharp lower bound for the cardinality of sumsets of
subsets of $\z^d$ confined to a hypercube, resolving in strong form a conjecture
that was made explicit by Becker, Ivanisvili, Krachun and
Madrid~\cite{Be-Iv-Kr-Ma} and had circulated in the folklore of the field for some time.
Specifically, for sets $A_j\subseteq \{0,1,2,\dots,m\}^d$ we show that
	\[|A_1+\dots+A_n|\;\geq\; (|A_1|\cdots|A_n|)^{1/p},
	\qquad p=\frac{n\log(m+1)}{\log(nm+1)},\]
with the exponent best possible. The only previously known sharp cases were $A_j\subseteq \{0,1\}^d$, for all $n\ge1$, and $A_j\subseteq \{0,1,2\}^d$ for $n=2$. We also prove a sharp inequality in the case when $A_j\subseteq\{0,1,\dots,m_j\}^d$ for different $m_j$.
We obtain the above inequality as a corollary of a stronger result on sup-convolution of functions on $\z^d$, whose proof is based on a novel mixed volume representation of a lattice path norm, together with a sharp one-dimensional functional inequality.
\end{abstract}

\maketitle

\section{Introduction}
In this work we resolve a conjecture about sharp lower bounds for the cardinality of sumsets of subsets of $\z^d$. If $A_1,\dots,A_n$ are finite subsets of $\z^d$, their Minkowski sum,
	$$
	A_1+A_2+\cdots + A_n :=\{a_1+a_2+\cdots+a_n :a_j\in A_j\},
	$$
must be large relative to the cardinalities $|A_j|$, and quantifying this
phenomenon sharply is a basic problem of additive combinatorics, with connections to convex geometry and Fourier analysis. We ask for the smallest exponent $p\geq 1$ for which the inequality
	\begin{equation}\label{ineq:conj}
	|A_1+\dots+A_n|\;\geq\;(|A_1|\cdots|A_n|)^{1/p},
	\end{equation}
holds true when each $A_j$ is confined to a hypercube $\{0,1,\dots,m\}^d$. Trivially the inequality holds with $p=n$. This problem has a long history. Its Boolean-cube version, when each $A_j\subseteq \{0,1\}^d$, goes back to Woodall~\cite{wo} and was settled, in the language of convolution measure algebras, by Hajela and Seymour~\cite{Ha} and by Landau, Logan and
Shepp~\cite{La-Lo-Sh}, with the best possible $p=\frac{n\log 2}{\log(n+1)}$ (see also Brown, Keane, Moran and Pierce~\cite{Br-Ke-Mo-Pi}
for the connection with Cantor measures and normal numbers). Closely related
compression and convex-geometry techniques were developed by Green and
Tao~\cite{Gr-Ta} in their work on the Freiman--Bilu theorem. Non-sharp bounds
for $n=2$ and $A_j\subseteq \{0,1,\ldots,m\}^d$ were obtained by Bourgain, Dilworth, Ford, Konyagin and Kutzarova~\cite{Bo-Di-Fo-Ko-Ku} in the course of constructing explicit RIP
matrices, and by Ivanisvili and Volberg~\cite{IvaVo} via Bellman-function
methods. Most recently, Becker, Ivanisvili, Krachun and
Madrid~\cite{Be-Iv-Kr-Ma} obtained the first sharp result beyond the Boolean
cube, for $n=2$ and $A_j\subseteq \{0,1,2\}^d$. We refer the reader to the references within these works, and recommend in particular~\cite[Appendix B]{Gr} and~\cite{TaoVu} for the additive combinatorics context.

In this manuscript we settle~\eqref{ineq:conj} in full generality. Throughout the paper
we write $[m]:=\{0,1,\ldots,m\}$, and we use $|A|$ for the cardinality of a finite set $A$. Let us define
	\begin{equation}\label{eq:pdef-intro}
	p_{n,m}:=\frac{n\log (m+1)}{m\log(n+1)}\,.
	\end{equation}
The following conjecture was stated in~\cite{Be-Iv-Kr-Ma}.

	\begin{conjecture}\label{conj:sets}
	If $A_j\subseteq [m]^d$ then~\eqref{ineq:conj} holds with $p=p_{nm,m}$.
	\end{conjecture}
This exponent is best possible since equality is attained when $A_1=\dots=A_n=[m]^d$.
The conjecture appears to have circulated in the folklore of the additive combinatorics field well before it was first stated explicitly in~\cite{Be-Iv-Kr-Ma}.

\subsection{Main Results} For two functions $f,g:\z^d\to\r_{\ge0}$ (with $\r_{\ge0}=[0,\infty)$) we
define their {sup-convolution} by
	\[f\bar{*}g(k) := \sup_{i+j=k} f(i)g(j).\]
The operation $\bar{*}$ is associative and commutative, so it can be iterated freely. If $f_j=\1_{A_j}$ is the indicator function of a set $A_j\subseteq\z^d$, then $f_1\bar{*}\cdots\bar{*}f_n=\1_{A_1+\dots+A_n}$ and $\|f_j\|_p=|A_j|^{1/p}$, where
	\[\|f\|_p:=\Big(\sum_{j}|f(j)|^p\Big)^{1/p}.\]
The main result of this paper is the following strong functional version of Conjecture~\ref{conj:sets}. 

\begin{theorem}\label{thm:main}
Let $f_1,f_2,\ldots,f_n:\z^d \to \r_{\ge0}$ be such that ${\rm supp}(f_j)\subseteq [m_j]^d$. Then the following inequality holds
    \begin{align}\label{ineq:main2}
    \|f_1\bar{*}\cdots \bar{*}f_n\|_1
    \geq \|f_1\|_{p_{M,m_1}} \cdots \|f_n\|_{p_{M,m_n}},
    \end{align}
    where $M=m_1+\dots+m_n$ and $p_{n,m}$ is defined in~\eqref{eq:pdef-intro}. Equality is attained when each $f_j$ is constant on $[m_j]^d$.
\end{theorem}

Specializing to indicator functions yields the following sharp sumset lower bound, settling
Conjecture~\ref{conj:sets} in the affirmative.

\begin{corollary}\label{thm:main1}
For any subsets $A_j \subseteq [m_j]^d$, $j=1,\dots,n$, setting $M=m_1+\dots+m_n$, we have
    \begin{align}\label{ineq:main}
    |A_1+\dots+A_n|\geq |A_1|^{1/p_{M,m_1}}\cdots|A_n|^{1/p_{M,m_n}}\,.
    \end{align}
Equality is attained if each $A_j=[m_j]^d$.
    \end{corollary}
    
As was mentioned above, in sharp form, the bound~\eqref{ineq:main} was previously known only when $A_j \subseteq \{0,1\}^d$ for all $n\geq 1$, or when $A_j \subseteq\{0,1,2\}^d$ for $n=2$. It is worth noting that, given our proof technique explained in the next subsection, both results above also hold in the case $\text{supp}(f_j) \subset S_j^d$, for arbitrary $S_j \subset \z$ with $m_j+1$ elements.

These results fit into a broader family of extremal inequalities for convolutions on discrete cubes. In this family, the same tension between cardinality, support and sharp constants appears in different flavors. For instance, replacing the $\ell^1$-norm of a sup-convolution by the $\ell^2$-norm or $\ell^\infty$-norm of an ordinary convolution leads to the {additive-energy} and {Sidon-set} problems about how concentrated can $f_1*\cdots*f_k$ be. These questions were studied on the
Boolean cube by Kane and Tao~\cite{Ka-Ta} (partitioning clusters) and by
de Dios, Greenfeld, Ivanisvili and Madrid~\cite{DGIM} (additive energies), and
were recently resolved in optimal form for $\{0,1\}^d$ by Gaitan and
Madrid~\cite{Ga-Ma}, with consequences for Sidon sets on the cube. On the other hand, their continuous counterparts are the classical autoconvolution and Sidon-set problems studied by Cilleruelo, Ruzsa and Vinuesa~\cite{Ci-Ru-Vi}. The geometry underlying our
proof, namely Brunn--Minkowski phenomena inside the cube, is also the subject
of the sharp isoperimetric inequalities on the hypercube of Beltran,
Ivanisvili and Madrid~\cite{Be-Iv-Ma}, while analytic, Prékopa--Leindler-type
routes to sumset cardinalities have been pursued by Matolcsi, Ruzsa, Shakan
and Zhelezov~\cite{MRSZ}. We view the present work as adding a new genuinely
convex-geometric tool to this toolbox, and we expect it to be useful in these
neighboring problems.

The technique we use to resolve Conjecture~\ref{conj:sets} appears to be new in this context. At its core is a translation of the discrete problem into the language of mixed volumes. This translation also brings the Aleksandrov--Fenchel inequality into play, while the sharp bound itself is reduced to a functional inequality for a single function. We sketch this strategy in more detail next, as it includes several results of independent interest.

\subsection{The strategy and further main results}
To every finitely supported function $f:\z\to \r_{\ge0}$ and every $n\ge1$ we associate the \emph{chain-of-cubes polytope}
\begin{equation} \label{eq:qbodydefrecap}
	Q_n(f) := \mathrm{conv}\Big(\bigcup_{x \in \z}\,[F(x),F(x+1)]^n\Big)\subseteq \r^n\,,\qquad\qquad 
	F(x)=\sum_{t<x}f(t)\,.        
\end{equation}
 In Figure~\ref{fig:pathnorm} we give examples of two-dimensional bodies $Q_2(f)$ (yellow and blue). Here $\mathrm{conv}(S)$ for $S\subseteq\r^n$ denotes the closed convex hull of the set $S$. Note that if $\text{supp}(f)=\{i_0,i_1,\dots,i_k\}$, $i_0<i_1<\dots<i_k$, then $Q_n(f)=Q_n(g)$ where $g:[k]\to\r_{\geq 0}$ is defined by $g(j)=f(i_j)$.

A classical theorem of Minkowski~\cite{Mk} states that for any collection $K_1,\dots,K_n$ of convex bodies in $\r^n$, the volume ${\rm vol}_{\r^n}(t_1K_1+\dots+t_nK_n)$, for $t_i\ge0$, is a homogeneous polynomial of degree $n$ in $t_1,\dots,t_n$ with nonnegative coefficients. The mixed volume, denoted $V(K_1,\dots,K_n)$, is a symmetric function of $n$-tuples of convex bodies, defined directly by the formula
    \[V(K_1,\dots,K_n) = \frac{1}{n!}\frac{\partial^n}{\partial{t_1}\dots\partial{t_n}} {\rm vol}_{\r^n}(t_1K_1+\dots+t_nK_n),\]
or by the so-called polarization formula (see~\cite[\S19.1]{Bu-Za})
	\[V(K_1,\dots,K_n) = \frac{1}{n!}\sum_{\ep\in\{0,1\}^n}(-1)^{n+\sum_i\ep_i}{\rm vol}_{\r^n}\Big(\sum_{i}\ep_iK_i\Big)\,.\]
In particular, $V(K):=V(K,K,\dots,K)={\rm vol}_{\r^n}(K)$. Mixed volumes have many nice properties, for example, they are nonnegative, monotone with respect to inclusion, linear with respect to positive Minkowski linear combinations, and define a valuation in each argument, that is,
    \begin{equation} \label{eq:mixedvaluation}
    V(K\cup L,K_2,\dots,K_n)=V(K,K_2,\dots,K_n)+V(L,K_2,\dots,K_n)-V(K\cap L,K_2,\dots,K_n)
    \end{equation}
for any convex bodies $K$ and $L$ such that $K\cup L$ is also convex.
One of the most important properties of mixed volumes is the Aleksandrov--Fenchel inequality (see~\cite[\S20]{Bu-Za}):
    \begin{equation} \label{eq:alexandrov-fenchel}
    V(K_1,K_2,K_3,\dots,K_n) \ge \sqrt{V(K_1,K_1,K_3,\dots,K_n)V(K_2,K_2,K_3,\dots,K_n)}\,.
    \end{equation}

Finally, for a function $g\colon \z^n\to \r$, we denote by $\|g\|_{\LL}$ the following \emph{lattice path norm}
	\begin{equation}\label{eq:latticenormrecap}
	\|g\|_{\LL}:=\sup_{k,\gamma}
	|g(\gamma(0))|+\dots+|g(\gamma(k))|\,,
	\end{equation}
where the supremum is over all $k \ge 0$ and all directed lattice paths $\gamma\colon[k]\to \z^n$ with steps $\gamma(j)-\gamma(j-1)\in \{e_1, \dots, e_n\}$, where $e_i$ are the standard basis vectors in $\z^n$. For us a lattice path will always mean a map $\gamma$ of this kind. The norm $\|\cdot\|_{\LL}$ is closely related to last passage percolation and the Robinson--Schensted--Knuth correspondence (see, e.g.,~\cite[Ch. 4.2, Ch. 5.2]{Rom}). Let us note the following trivial inequality, connecting the $\LL$-norm to sup-convolution:
	\begin{equation}\label{ineq:trivialLlwouerbound}
	\|f_1\bar{*}\cdots\bar{*}f_n\|_1 \geq \|f_1 \otimes \dots \otimes f_n\|_{\LL}.
	\end{equation}

Our main geometric result is the following theorem connecting the above notions.

\begin{theorem} \label{thm:mvformula}
For any finitely supported functions $f_1,\dots,f_n\colon\z\to\r_{\ge0}$ we have
    \begin{equation}\label{eq:mvformula}
    V(Q_n(f_1),\dots,Q_n(f_n))
     = \|f_1\otimes \dots \otimes f_n\|_{\LL}\,.
    \end{equation}
\end{theorem}
\noindent For $n=2$ Theorem \ref{thm:mvformula} takes a particularly nice geometric form, illustrated in Figure~\ref{fig:pathnorm}.

\begin{figure}[h]
	\begin{center}
		\definecolor{flaggreen}{HTML}{009B3A}
		\definecolor{flagyellow}{HTML}{FFDF00}
		\definecolor{flagblue}{HTML}{0057B7}
		\begin{tikzpicture}[scale=0.36]
			\fill[fill=flagyellow!70, opacity=0.7] (0,0) -- (1,0) -- (4,1) -- (6,4) -- (6,6) -- (4,6) --
			(1,4) -- (0,1) -- cycle;
			\fill[fill=flagyellow!85, opacity=0.7] (0,0) rectangle (1,1);
			\fill[fill=flagyellow!85, opacity=0.7] (1,1) rectangle (4,4);
			\fill[fill=flagyellow!85, opacity=0.7] (4,4) rectangle (6,6);
			
			\fill[fill=flagblue!70, opacity=0.7] (6,6) -- (8,6) -- (11,8) -- (15,12) -- (15,15) --
			(12,15) --
			(8,11) -- (6,8) -- cycle;
			\fill[fill=flagblue!85, opacity=0.7] (6,6) rectangle (8,8);
			\fill[fill=flagblue!85, opacity=0.7] (8,8) rectangle (11,11);
			\fill[fill=flagblue!85, opacity=0.7] (11,11) rectangle (12,12);
			\fill[fill=flagblue!85, opacity=0.7] (12,12) rectangle (15,15);
			
			\fill[fill=flaggreen!70, opacity=0.7] (0,1) -- (1,4) -- (4,6) -- (6,6) -- (6,8) -- (8,11)
			-- (12,15) -- (10,15) -- (7,13) -- (3,9) -- (1,6) -- (0,3) -- cycle;
			
			\fill[fill=flaggreen!70, opacity=0.7] (6,0) rectangle (8,4);
			\fill[fill=flaggreen!70, opacity=0.7] (8,1) rectangle (15,4);
			\fill[fill=flaggreen!70, opacity=0.7] (12,4) rectangle (15,6);
			\foreach \x in {6,8,11,12,15}
			\draw[very thin, color=gray!30] (\x,0) -- (\x,15);
			\foreach \x in {0,1,4}
			\draw[very thin, color=gray!30] (\x,0) -- (\x,6);
			\foreach \y in {6,8,11,12,15}
			\draw[very thin, color=gray!30] (6,\y) -- (15,\y);
			\foreach \y in {0,1,4,6}
			\draw[very thin, color=gray!30] (0,\y) -- (15,\y);
			\foreach \x in {6,8,11,12,15}
			\draw[very thin, color=black!80] (\x,0) -- (\x,6);
			\foreach \y in {0,1,4,6}
			\draw[very thin, color=black!80] (6,\y) -- (15,\y);
			\foreach \a/\b/\c/\d in {0/0/1/0, 1/0/4/1, 4/1/6/4, 6/4/6/6}
			\draw[dashed, ->, >=stealth, thick, black!80] (\a,\b) -- (\c,\d);
			\foreach \a/\b/\c/\d in {6/6/8/6, 8/6/11/8, 11/8/15/12, 15/12/15/15}
			\draw[dashed, ->, >=stealth, thick, black!80] (\a,\b) -- (\c,\d);
		\end{tikzpicture}
	\end{center}
	\caption{A geometric interpretation of Theorem~\ref{thm:mvformula} for
		$n=2$. The yellow and blue polygons are $Q_2(f_i)$. The upper left green region is a
		component of the complement of $Q_2(f_1)\cup Q_2(f_2)$ in the Minkowski sum
		$Q_2(f_1)+Q_2(f_2)$ and its area is the mixed volume $V(Q_2(f_1),Q_2(f_2))$. The green
		figure on the bottom right shows the path of maximal area from the lower left to the upper right rectangle. The theorem asserts that the two green regions have equal
		areas.}\label{fig:pathnorm}
\end{figure}
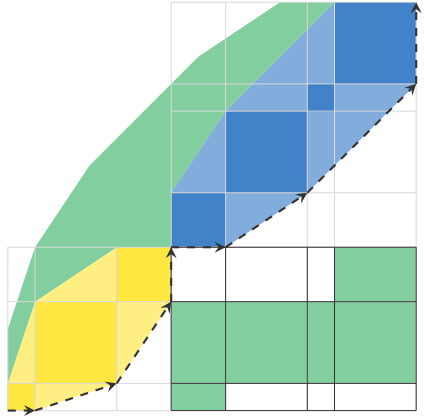

\noindent Since mixed volumes satisfy the Aleksandrov--Fenchel inequality, we see that an analogous property holds for $\|\cdot\|_{\LL}$.
\begin{corollary} 
For any functions $f_1,\dots,f_n\colon\z\to\r_{\ge0}$ we have
    \begin{equation}
    \|f_1\otimes f_2\otimes \dots \otimes f_n\|_{\LL}\ge 
    \sqrt{\|f_1\otimes f_1\otimes f_3\otimes \dots \otimes f_n\|_{\LL}\cdot \|f_2\otimes f_2\otimes f_3\otimes\dots \otimes f_n\|_{\LL}}\,.
    \end{equation}
\end{corollary}
Our last main result is a lower bound on the mixed volume of $Q_n(f)$ in terms of $\ell^p$-norms.
\begin{theorem}\label{thm:VolLowerbound}
For any functions $f_j:\z\to \r_{\ge0}$ with $\text{supp}(f_j) \subseteq [m_j]$, $j=1,\dots,n$, we have
    \begin{equation} \label{eq:mvpnorm}
   V(Q_n(f_1),\dots,Q_n(f_n)) \geq \|f_1\|_{p_{M,m_1}}\cdots \|f_n\|_{p_{M,m_n}},
	\end{equation}
where $M=m_1+\dots+m_n$ and $p_{n,m}$ is defined in~\eqref{eq:pdef-intro}.
\end{theorem}

With Theorems \ref{thm:mvformula} and \ref{thm:VolLowerbound} in hand, we can now give a short proof of Theorem \ref{thm:main}.

\begin{proof}[\bf Proof of Theorem \ref{thm:main}] By the dimension compression argument of ~\cite[Thm. 2.1]{Be-Iv-Kr-Ma} we only need to prove Theorem~\ref{thm:main} in the one-dimensional case $d=1$. We can now apply the trivial lower bound \eqref{ineq:trivialLlwouerbound} in conjunction with Theorem  \ref{thm:mvformula} and Theorem \ref{thm:VolLowerbound}:
\begin{align}\label{ineq:thechain}
	\boxed{
		\|f_1\bar{*}\cdots\bar{*}f_n\|_1    \overset{\eqref{ineq:trivialLlwouerbound}}{\geq} \|f_1 \otimes \dots \otimes f_n\|_{\LL} \overset{\eqref{eq:mvformula}}{=} V(Q_n(f_1),\dots,Q_n(f_n))  \overset{\eqref{eq:mvpnorm}}{\geq} \|f_1\|_{p_{M,m_1}} \cdots \|f_n\|_{p_{M,m_n}}.
	}
	\end{align}
	This finishes the proof.
\end{proof}

A posteriori we noticed that our proof technique bears some resemblance to Stanley's ~\cite[Thm.~5]{Sta} classical use of the Aleksandrov--Fenchel inequalities to prove log-concavity for certain combinatorial quantities associated with partially ordered sets by attaching certain order polytopes to them. However, in practice, our geometric construction and the way the inequality is used are very different in nature. 

\subsection{Organization of the paper}
In Section~\ref{sec:ineqA} we prove Theorem \ref{thm:mvformula} and in Section \ref{sec:ineqB} we prove Theorem \ref{thm:VolLowerbound}. These sections can be read almost independently. 
For completeness, in the Appendix we reproduce the dimension compression argument used in the proof of Theorem~\ref{thm:main}.

\section{Chains of cubes}
\label{sec:ineqA}

This section is devoted to the proof of Theorem~\ref{thm:mvformula}.

\subsection{Computation of volumes for a more general family}
Let us consider a slightly more general family of convex bodies. Given a finite collection $\I$ of closed intervals in~$\r$ we define 
    \begin{equation} \label{eq:qnIdef}
    Q_n(\I) := \mathrm{conv}\Big(\bigcup_{I\in\I}I^n\Big).
    \end{equation}
Clearly,~\eqref{eq:qbodydefrecap} is a special case of~\eqref{eq:qnIdef}. If $\I=\{I_0,\dots,I_m\}$, using the fact that cubes are convex, we get an alternative description
    \begin{equation} \label{eq:qndef2}
    Q_n(\I) = \bigcup_{\lambda\in\Delta^m}\lambda_0 I_0^n+\dots+\lambda_mI_m^n
    = \bigcup_{\lambda\in\Delta^m}(\lambda_0 I_0+\dots+\lambda_mI_m)^n\,,
    \end{equation}
where $\Delta^m:=\{\lambda\in\r_{\ge0}^{[m]}\colon\lambda_0+\dots+\lambda_m=1\}$. This implies the following property:
    \begin{equation}\label{eq:q2-qn}
    x\in Q_n(\I) \quad \Leftrightarrow \quad 
    (\max x,\,\min x) \in Q_2(\I)\,.
    \end{equation}
Indeed, taking projections the ``$\Rightarrow$'' implication is immediate. If $(\max x,\min x)\in Q_2(\I)$, then for some choice of $\lambda\in\Delta^m$
from~\eqref{eq:qndef2} we get that both $\min x$ and $\max x$ are in the interval $I=\lambda_0I_0+\dots+\lambda_mI_m$, and thus $x\in I^n\subseteq Q_n(\I)$, proving ``$\Leftarrow$''.

The benefit of working with this more general family is the fact that the class of convex sets $\{Q_n(\I)\}_{\I}$ is closed under Minkowski sums:
    \[Q_n(\I)+Q_n(\J)=Q_n(\{I+J\colon I\in\I, J\in\J\}).\]
The following lemma gives a formula for the volume of $Q_n(\I)$.
\begin{lemma} \label{lem:vqn}
Let $\{I_0,\dots,I_m\}$ be the minimal set of intervals with $Q_2(I_0,\dots,I_m)=Q_2(\I)$, and assume that $I_0,\dots,I_m$ are ordered by their right endpoint in increasing order. Then
	\begin{equation} \label{eq:vqn}
 V(Q_n(\I)) = |I_0|^n
	 + \sum_{j=1}^{m}(b_{j}-b_{j-1})\frac{|I_{j}|^n-|I_{j-1}|^n}{|I_{j}|-|I_{j-1}|}\,,
    \end{equation}
where $I_j=[a_j,b_j]$.
\end{lemma}
(If $|I_j|=|I_{j-1}|$ the expression $\frac{|I_j|^n-|I_{j-1}|^n}{|I_j|-|I_{j-1}|}$ above should be interpreted as $n|I_j|^{n-1}$.)
\begin{proof}
Using~\eqref{eq:q2-qn} and integrating along the fibers of the map $x\mapsto (\max x, \min x)$ gives
	\[V(Q_n(\I)) = n(n-1)\int_{Q_2(\I)\cap\{x\ge y\}}(x-y)^{n-2}dx dy\,.\]
If we set $\omega = n(x-y)^{n-1}dx$, then $d\omega = n(n-1)(x-y)^{n-2}dx\wedge dy$, and by Green's formula 
	\begin{align}\label{id:volintegral}
	    V(Q_n(\I)) = \int_{\gamma}n(x-y)^{n-1}dx,
	\end{align}
where $\gamma$ is the lower boundary of the polygon $Q_2(\I)$ (the diagonal $x=y$ doesn't appear since the integrand vanishes there). The condition of the lemma implies that $\gamma$ is a polygonal curve with vertices $(a_0,a_0)$, $(b_j,a_j)$, $j=0,\dots,m$, and $(b_m,b_m)$. To get~\eqref{eq:vqn} it remains to note that for any linear function $l$ we have $\int_{a}^{b}n\,l(u)^{n-1}du = (b-a)\frac{l(b)^n-l(a)^n}{l(b)-l(a)}$. 
\end{proof}
This lemma gives the following useful formula for a function $f:[m]\to\r_{\geq 0}$:
	\begin{equation} \label{eq:volumeqf}
	V(Q_n(f)) =  f(0)^n 
    + \sum_{j=1}^{k}\frac{f(i_j)^n-f(i_{j-1})^n}{f(i_j)-f(i_{j-1})}\sum_{l=i_{j-1}+1}^{i_j}f(l),
	\end{equation}
   where $0=i_0<i_1<...<i_k=m$ are chosen such that $k$ is minimal and
    \[
    Q_2(f)=\mathrm{conv}\left(\bigcup_{j=0}^k \bigg[ \sum_{i< i_{j}} f(i), \sum_{i\leq i_{j}} f(i)\bigg]^2\right).
    \]

\subsection{Edge vector parameterization}
We will use the following standard parameterization of convex polygons (see~\cite[\S27]{Lyu}, and~\cite[Ch.4]{Sch} for a more general notion). To a convex polygon $P$ with vertices $p_1,\dots,p_k$, ordered counterclockwise, we associate a cycle of edge vectors $c(P)=(v_1,\dots,v_k)$, where $v_i=p_{i+1}-p_i$, with indices labeled modulo $k$. The cycle is well-defined up to cyclic permutations, it satisfies $\sum_{j}v_j=0$ and the vectors $v_i$ are oriented cyclically counterclockwise by their angle. Conversely, any collection $v_1,\dots,v_k\in\r^2$ of non-zero vectors, ordered cyclically and satisfying $v_1+\dots+v_k=0$, corresponds to a unique convex polygon (up to translation). The important property for us is that
	\[c(tP) = (tv_1,tv_2,\dots,tv_k) ,\quad t>0,\]
and $c(P+Q)$ is the union of $c(P)$ and $c(Q)$ with edge vectors shuffled into cyclic order and pairs of vectors pointing in the same direction merged (added) together.

Let $v_0,\dots,v_{m+1}$ be the edge vectors of $Q_2(\I)$ in the first quadrant, ordered counterclockwise (see the dashed arrows in Figure \ref{fig:pathnorm}). We will denote $c_+(Q_2(\I))=(v_0,\dots,v_{m+1})$. If we write $v_i=(x_i,y_i)$, then from the construction of $Q_2(\I)$ it is clear that $y_0=x_{m+1}=0$. In this notation $|I_j|=\sum_{i=0}^{j}(x_{i}-y_{i})$ and~\eqref{eq:vqn} becomes
\begin{equation} \label{eq:vqn2}
	V(Q_n(\I)) = x_{0}^n
	+ \sum_{j=1}^{m}\frac{x_{j}}{x_{j}-y_{j}}\Big(\big(\sum_{i\le j}(x_{i}-y_{i})\big)^n-\big(\sum_{i<j}(x_{i}-y_{i})\big)^n\Big)\,.
\end{equation}

We will also need the description of edge vectors for $Q_2(f)$.
Let $f\colon\z\to\r_{\ge0}$ be supported in $\{0,\dots,M\}$. Then $c_{+}(Q_2(f))=(v_0,\dots,v_{m+1})$, where
\begin{equation} \label{eq:q2f_sides}
	v_j = \Big(\sum_{i_{j-1}<i\le i_j}f(i),\sum_{i_{j-1}\le i<i_{j}}f(i)\Big)\,,
	\qquad 0\le j\le m+1\,,
\end{equation}
for some sequence of indices $0=i_0<i_1<\dots<i_m=M$ with $i_{-1}=-1$ and $i_{m+1}=M+1$.

\subsection{Proof of Theorem~\ref{thm:mvformula}}
We are now ready to prove Theorem~\ref{thm:mvformula}. The proof is split into two parts: first we will show that $V(Q_n(f_1),\dots,Q_n(f_n))$ equals the sum of $f=f_1\otimes\dots\otimes f_n$ over a specific path, and then we will show that this path realizes the maximal path sum. 

\textbf{The maximal path.} Let us first explicitly describe the maximal path and introduce some notation. Without loss of generality, assume that $f_l$ is supported in $\{0,1,\dots,M_l\}=[M_l]$ and let $M=M_1+\dots+M_n$. Let $c_+(Q_2(f_l)) = (v_0^{(l)},\dots,v_{m_l+1}^{(l)})$ be the edge vectors of $Q_2(f_l)$ in the first quadrant, and let $0=i_0^{(l)}<\dots<i_{m_l}^{(l)}=M_l$ be the corresponding sequence of indices defined as in~\eqref{eq:q2f_sides}, with $i_{-1}^{(l)}=-1$ and $i_{m_l+1}^{(l)}:=M_l+1$.

These indices define a map $I\colon [m_1]\times \dots \times [m_n]\to [M_1]\times \dots \times [M_n]$, given by
	\[I(u) = (i_{u_1}^{(1)}, \dots, i_{u_n}^{(n)})\,.\]
For any lattice path $\sigma$ in $[m_1]\times \dots \times [m_n]$, we define its extension to a lattice path $I(\sigma)$ in $[M_1]\times \dots \times [M_n]$ by connecting the sequence of vertices $I(w)$ for $w \in \sigma$ with straight line segments, i.e., for any consecutive points $I(w)$ and $I(w')$ with step $I(w')-I(w)=\mu e_k$ we insert $\mu-1$ additional points between $I(w)$ and $I(w')$ with steps $e_k$.

As established above, edge vectors $v_1^{(l)},\dots,v_{m_l}^{(l)}$ of $Q_2(f_l)$ in the first quadrant are ordered by their strictly increasing slopes. Let us consider the union of these edge vectors for all $l=1,\dots,n$, and sort them by their slopes in (possibly non-strictly) increasing order. This naturally defines a path $\sigma_{\mathrm{max}}$ from $(0,\dots,0)$ to $(m_1,\dots,m_n)$ by taking a step $e_l$ whenever the next vector in the sorted sequence is of the form $v_j^{(l)}$. The maximal path $\gamma_{\mathrm{max}}$ is then defined as the extension $I(\sigma_{\mathrm{max}})$ of $\sigma_{\mathrm{max}}$.

To simplify the arguments below, we will make a general position assumption for $f_l$, so that all pairs of nontrivial edge vectors are non-collinear, no edge vector is parallel to the diagonal, and no two distinct paths have the same sum. It does not reduce generality, since the set of $n$-tuples of functions (with fixed support) satisfying conditions in each lemma below is a closed set in the natural topology on the corresponding Euclidean space, and the subset of $n$-tuples of functions in general position is open and dense.

\begin{lemma} \label{lem:mixedvolumesmaller}
	In the above notation we have
	\begin{equation}
		V(Q_n(f_1),\dots,Q_n(f_n)) = \sum_{j=0}^{M}f(\gamma_{\mathrm{max}}(j))\,.
	\end{equation}
\end{lemma}
\begin{proof}
Let $c_+(Q_2(f_l)) = (v_0^{(l)},\dots,v_{m_l+1}^{(l)})$ be the edge vectors of $Q_2(f_l)$ in the first quadrant, and let $0=i_0^{(l)}<\dots<i_{m_l}^{(l)}=M_l$ be the corresponding sequence of indices defined in~\eqref{eq:q2f_sides}, with $i_{-1}^{(l)}:=-1$ and $i_{m_l+1}^{(l)}:=M_l+1$.

For the Minkowski sum, we have $c_+(t_1Q_2(f_1)+\dots+t_nQ_2(f_n))=(w_0,\dots,w_{N+1})$, where $N=m_1+\dots+m_n$ and 
\begin{align*}
	w_0 =& \;t_1v_0^{(1)}+\dots+t_nv_0^{(n)}\,,\\
	w_{N+1} =& \;t_1v_{m_1+1}^{(1)}+\dots+t_nv_{m_n+1}^{(n)}\,,\\
	\{w_1,\dots,w_N\} =& \;\textstyle\bigcup_{l=1}^n\{t_lv_i^{(l)}\colon 1\le i \le m_l\}\,.
\end{align*}
The vectors $w_j$ for $1\le j\le N$ are ordered by their strictly increasing slopes. We encode the order in which the individual vectors $t_k v_i^{(k)}$ appear using the lattice path $\sigma_{\mathrm{max}}$ defined above, whose sequence of vertices we denote by $\sigma_0,\dots,\sigma_N$. Note that $\sigma_0 = (0,\dots,0)$ and
	\[\sigma_j-\sigma_{j-1}=e_k\qquad \Leftrightarrow \qquad w_j = t_k v_{\sigma_{j,k}}^{(k)},\]
where $\sigma_{j,k}$ is the $k$-th coordinate of $\sigma_j$.
With this notation we have
	\[A_j:=\sum_{\nu\le j}(w_{\nu,1}-w_{\nu,2}) = \sum_{l=1}^{n}t_l f_l(i_{\sigma_{j,l}}^{(l)})\,.\]
Note that the coefficient of $t_1\cdots t_n$ in $(a_1t_1+\dots+a_nt_n)^n$ is $n!\,a_1\cdots a_n$. Therefore, setting $\eta_j = I(\sigma_j)$, we see that
	\[\frac{1}{n!}\frac{\partial^n}{\partial{t_1}\dots\partial{t_n}}A_j^n = f(\eta_j)\,.\]
Together with~\eqref{eq:vqn2}, this implies that the mixed volume $V(Q_n(f_1),\dots,Q_n(f_n))$ equals
	\begin{equation} \label{eq:mvcalculation1}
	f(0,\dots,0)+\sum_{j=1}^{N}\frac{w_{j,1}}{w_{j,1}-w_{j,2}}\big(f(\eta_j)-f(\eta_{j-1})\big)\,.
	\end{equation}
Let $k$ be the index such that $\sigma_j-\sigma_{j-1}=e_k$. By~\eqref{eq:q2f_sides}, we have $w_{j,1} = t_k \sum_{i=\eta_{j-1,k}+1}^{\eta_{j,k}} f_k(i)$ and $w_{j,1}-w_{j,2} = t_k \big(f_k(\eta_{j,k}) - f_k(\eta_{j-1,k})\big)$. Thus,
	\[\frac{w_{j,1}}{w_{j,1}-w_{j,2}}\big(f(\eta_j)-f(\eta_{j-1})\big) = 
	\sum_{i=\eta_{j-1,k}+1}^{\eta_{j,k}} f(\eta_{j-1} + (i - \eta_{j-1,k}) e_k)\,.\]
Plugging this into~\eqref{eq:mvcalculation1}, together with the definition of $\gamma_{\mathrm{max}}$ gives the claim.
\end{proof}

\begin{lemma} \label{lem:mixedvolumegreater}
	For $f_j\colon\z\to\r_{\ge0}$ let $f=f_1\otimes \dots\otimes f_n$ and define $\gamma_{\mathrm{max}}$ as above. Then
	\[\sum_{j=0}^{M}f(\gamma(j)) \le \sum_{j=0}^{M}f(\gamma_{\mathrm{max}}(j))\,,\]
	for any lattice path $\gamma\colon[M]\to [M_1]\times \dots \times [M_n]$.
\end{lemma}
\begin{proof}
Recall that $f_l$ is supported in $[M_l]$. Let us denote for any path $\gamma$ the sum of $f$ over its vertices by $S(\gamma)$, so the inequality that we need to prove is $S(\gamma)\le S(\gamma_{\mathrm{max}})$. Let us also recall the notation $F_l(x)=\sum_{t<x}f_l(t)$. For $n=1$ the claim is trivial since there is only one lattice path, so we assume that $n\ge2$.
	
\noindent\textbf{Case $n=2$.} We proceed by induction on $M_1+M_2$. The base cases $M_1=0$ or $M_2=0$ are trivial. Thus we may assume that $M_1,M_2>0$.

Suppose that neither $Q_2(f_1)$ nor $Q_2(f_2)$ has internal vertices on their lower boundaries (i.e., if $(F_l(j+1),F_l(j))$ lies on the boundary of $Q_2(f_l)$, then $j=0$ or $j=M_l$). In particular, each $Q_2(f_i)$ has only one non-trivial edge vector in the first quadrant. Let $\lambda_1^{-1}$ and $\lambda_2^{-1}$ be the slopes of these edges. We have
	\[\lambda_1=\frac{F_1(M_1+1)-f_1(0)}{F_1(M_1)}\,, \qquad
	  \lambda_2=\frac{F_2(M_2+1)-f_2(0)}{F_2(M_2)}\,.\]
We fix $f_l(i)$, $i\le M_l-1$ and treat the difference $S(\gamma_{\mathrm{max}}) - S(\gamma)$ as a function of $\lambda_1,\lambda_2$. Fixing $\gamma$ and $\gamma_{\mathrm{max}}$ corresponds to the restrictions $\lambda_i\ge c_i$ for some $c_i\ge0$ (determined in terms of the slopes of $Q_2(f_i)$). By symmetry, we may assume $\lambda_1 \ge \lambda_2$. Then depending on $\gamma$ we have either
	\[S(\gamma_{\mathrm{max}}) - S(\gamma) = \lambda_1F_1(M_1)(F_2(M_2)-F_2(l))+C\]
or 
	\[S(\gamma_{\mathrm{max}}) - S(\gamma) = F_2(M_2)\Big(\lambda_1F_1(M_1) - \lambda_2(F_1(M_1)-F_1(k))\Big)+C\,,\]
for some $0\le k\le M_1$, $0\le l \le M_2$. In either case, on the domain $\{(\lambda_1,\lambda_2)\colon \lambda_i\ge c_i, \lambda_1\ge \lambda_2\}$ the minimum is attained at one of the vertices, and hence by changing $\lambda_i$ we can make the inequality stronger and attain one of the boundary conditions $\lambda_1=c_1$ or $\lambda_2=c_2$. But $\lambda_i=c_i$ implies that either the support of $f_i$ decreases (in which case we are done by induction) or $Q_2(f_i)$ has an internal vertex.

Now suppose there is an internal vertex on the lower boundary of $Q_2(f_1)$. This means that for some index $0 < i_0 < M_1$ we can write $Q_2(f_1) = K \cup L$ where $K$ is $ Q_2(f_1')$ and $L$ is $Q_2(f_1'')$ shifted by $(F_1(i_0),F_1(i_0))$, where $f_1'$ is the restriction of $f_1$ to $\{0,\dots,i_0\}$ and $f_1''$ is the restriction of $f_1$ to $\{i_0,\dots,M_1\}$ (see Figure~\ref{fig:q2_decomposition}). Since $K\cap L$ is a square of size $f_1(i_0)$, by the valuation property of mixed volumes~\eqref{eq:mixedvaluation}, we have
	\[ V(Q_2(f_1), Q_2(f_2)) = V(Q_2(f_1'), Q_2(f_2)) + V(Q_2(f_1''), Q_2(f_2)) - V([0, f_1(i_0)]^2, Q_2(f_2))\,. \]
By a simple computation (or invoking Lemma~\ref{lem:mixedvolumesmaller}), the last term is
	\[V([0, f_1(i_0)]^2, Q_2(f_2)) = f_1(i_0) \sum_{j=0}^{M_2} f_2(j)\,.\] 
	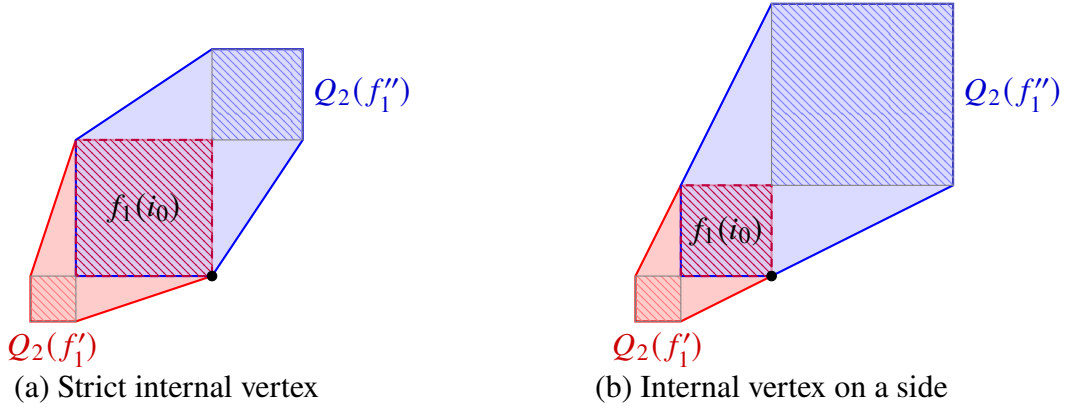
\begin{figure}[ht]
	\centering
	\begin{tikzpicture}
		
		\begin{scope}[xshift=0cm, scale=0.6]
			\def\squareA{1}
			\def\squareB{3}
			\def\squareC{2}
			\coordinate (A1) at (0,0);
			\coordinate (A2) at (\squareA,0);
			\coordinate (A3) at (0,\squareA);
			\coordinate (A4) at (\squareA,\squareA);
			\coordinate (B1) at (\squareA,\squareA);
			\coordinate (B2) at (\squareA+\squareB,\squareA);
			\coordinate (B3) at (\squareA,\squareA+\squareB);
			\coordinate (B4) at (\squareA+\squareB,\squareA+\squareB);
			\coordinate (C1) at (\squareA+\squareB,\squareA+\squareB);
			\coordinate (C2) at (\squareA+\squareB+\squareC,\squareA+\squareB);
			\coordinate (C3) at (\squareA+\squareB,\squareA+\squareB+\squareC);
			\coordinate (C4) at (\squareA+\squareB+\squareC,\squareA+\squareB+\squareC);

			\fill[red!20] (A1) -- (A3) -- (B3) -- (B4) -- (B2) -- (A2) -- cycle;
			\draw[thick, red] (A1) -- (A3) -- (B3) -- (B4) -- (B2) -- (A2) -- cycle;
			
			\fill[blue!20, opacity=0.7] (B1) -- (B3) -- (C3) -- (C4) -- (C2) -- (B2) -- cycle;
			\draw[thick, blue] (B1) -- (B3) -- (C3) -- (C4) -- (C2) -- (B2) -- cycle;
			
			\fill[pattern color=purple, pattern=north west lines] (B1) rectangle (B4);
			\draw[thick, purple, dashed] (B1) rectangle (B4);
			\node[inner sep=1pt, rounded corners=1pt] at ($(B1)!0.5!(B4)$) {$f_1(i_0)$};
			
			\fill[pattern color=red!50, pattern=north west lines] (A1) rectangle (A4);
			\fill[pattern color=blue!50, pattern=north west lines] (C1) rectangle (C4);
			\draw[gray] (A1) rectangle (A4);
			\draw[gray] (C1) rectangle (C4);
			
			\filldraw[black] (B2) circle (3pt);
			
			\node at (3, -1.5) {(a) Strict internal vertex};
			\node[red!80!black, below] at (0.5, 0) {$Q_2(f'_1)$};
			\node[blue!80!black, right] at (6, 5) {$Q_2(f''_1)$};
		\end{scope}
		
		\begin{scope}[xshift=8cm, scale=0.6]
			\def\squareA{1}
			\def\squareB{2}
			\def\squareC{4}
			\coordinate (A1) at (0,0);
			\coordinate (A2) at (\squareA,0);
			\coordinate (A3) at (0,\squareA);
			\coordinate (A4) at (\squareA,\squareA);
			\coordinate (B1) at (\squareA,\squareA);
			\coordinate (B2) at (\squareA+\squareB,\squareA);
			\coordinate (B3) at (\squareA,\squareA+\squareB);
			\coordinate (B4) at (\squareA+\squareB,\squareA+\squareB);
			\coordinate (C1) at (\squareA+\squareB,\squareA+\squareB);
			\coordinate (C2) at (\squareA+\squareB+\squareC,\squareA+\squareB);
			\coordinate (C3) at (\squareA+\squareB,\squareA+\squareB+\squareC);
			\coordinate (C4) at (\squareA+\squareB+\squareC,\squareA+\squareB+\squareC);

			\fill[red!20] (A1) -- (A3) -- (B3) -- (B4) -- (B2) -- (A2) -- cycle;
			\draw[thick, red] (A1) -- (A3) -- (B3) -- (B4) -- (B2) -- (A2) -- cycle;
			
			\fill[blue!20, opacity=0.7] (B1) -- (B3) -- (C3) -- (C4) -- (C2) -- (B2) -- cycle;
			\draw[thick, blue] (B1) -- (B3) -- (C3) -- (C4) -- (C2) -- (B2) -- cycle;
			
			\fill[pattern color=purple, pattern=north west lines] (B1) rectangle (B4);
			\draw[thick, purple, dashed] (B1) rectangle (B4);
			\node[inner sep=1pt, rounded corners=1pt] at ($(B1)!0.5!(B4)$) {$f_1(i_0)$};
			
			\fill[pattern color=red!50, pattern=north west lines] (A1) rectangle (A4);
			\fill[pattern color=blue!50, pattern=north west lines] (C1) rectangle (C4);
			\draw[gray] (A1) rectangle (A4);
			\draw[gray] (C1) rectangle (C4);
			
			\filldraw[black] (B2) circle (3pt);
			
			\node at (3, -1.5) {(b) Internal vertex on a side};
			\node[red!80!black, below] at (0.5, 0) {$Q_2(f'_1)$};
			\node[blue!80!black, right] at (7, 5) {$Q_2(f''_1)$};
		\end{scope}
		
	\end{tikzpicture}
	\caption{Decomposition of $Q_2(f_1)$ into $Q_2(f_1') \cup Q_2(f_1'')$.}
	\label{fig:q2_decomposition}
\end{figure}

Let $j_{\mathrm{in}}\le j_{\mathrm{out}}$ be the $j$-coordinates where $\gamma$ enters and leaves the column $i=i_0$. We can construct two new valid lattice paths: $\gamma'$ for $f_1' \otimes f_2$ which goes from $(0,0)$ to $(i_0, M_2)$ by following $\gamma$ and then moving straight up from $(i_0, j_{\mathrm{out}})$, and $\gamma''$ for $f_1'' \otimes f_2$ which goes from $(i_0,0)$ to $(M_1, M_2)$ by moving straight up to $(i_0, j_{\mathrm{in}})$ and then following $\gamma$ (see Figure~\ref{fig:path_split}). Since the segments $[j_{\mathrm{in}}, M_2]$ and $[0, j_{\mathrm{out}}]$ cover $[0, M_2]$ with an overlap of exactly $[j_{\mathrm{in}}, j_{\mathrm{out}}]$, we get the following identity:
	\[ S(\gamma) = S(\gamma') + S(\gamma'') -  f_1(i_0)\sum_{j=0}^{M_2} f_2(j).\]
By the inductive hypothesis, $S(\gamma') \le V(Q_2(f_1'), Q_2(f_2))$ and $S(\gamma'') \le V(Q_2(f_1''), Q_2(f_2))$. Subtracting the intersection term gives $S(\gamma) \le V(Q_2(f_1), Q_2(f_2))$, completing the inductive step (using Lemma~\ref{lem:mixedvolumesmaller}). By symmetry, the same holds if $Q_2(f_2)$ has an internal vertex. This proves the claim for $n=2$.

	\begin{figure}[ht]
	\centering
	\begin{tikzpicture}[scale=0.85]
		\draw[help lines, color=gray!40] (0,0) grid (11,6);
		
		\draw[dashed, thick, red!50!blue] (5,-0.5) node[below, text=black] {$i_0$} -- (5,6.5);
		
		\node[below] at (0,0) {$(0,0)$};
		\node[below] at (11,0) {$(M_1,0)$};
		\node[above] at (0,6) {$(0,M_2)$};
		\node[above] at (11,6) {$(M_1,M_2)$};
		
		\draw[thick, black] (0,0) -- (2,0) -- (2,1) -- (4,1) -- (5,1) -- (5,4) -- (6,4) -- (8,4) -- (8,5) -- (10,5) -- (10,6) -- (11,6);
		\node at (2.5, 0.4) {$\gamma$};
		
		\draw[thick, red, transform canvas={xshift=-2.5pt, yshift=2.5pt}] (0,0) -- (2,0) -- (2,1) -- (4,1) -- (5,1) -- (5,4) -- (5,6);
		\node[red, left] at (4.9, 5) {$\gamma'$};
		
		\draw[thick, blue!80, transform canvas={xshift=2.5pt, yshift=-2.5pt}] (5,0) -- (5,1) -- (5,4) -- (6,4) -- (8,4) -- (8,5) -- (10,5) -- (10,6) -- (11,6);
		\node[blue!80, right] at (5.1, 0.5) {$\gamma''$};
		
		\filldraw (5,1) circle (2.5pt);
		\filldraw (5,4) circle (2.5pt);
	\end{tikzpicture}
	\caption{Splitting the path $\gamma$ into $\gamma'$ and $\gamma''$ at the vertical line $x=i_0$.}
	\label{fig:path_split}
\end{figure}
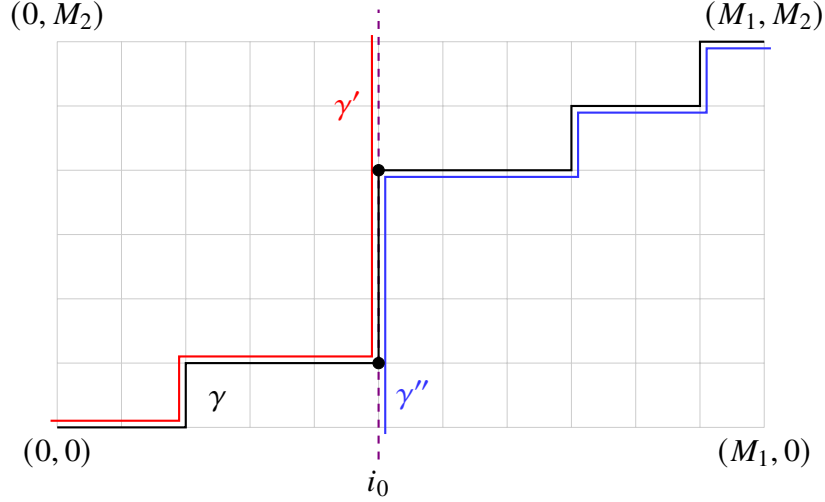

\noindent\textbf{Case $n > 2$.} Consider a lattice path $\gamma$ maximizing the sum. Projecting $\gamma$ onto the first $n-1$ coordinates yields, after removing repeating points, a lattice path $\pi$ from $0$ to $(M_1,\dots,M_{n-1})$. Defining $g(k) = (f_1\otimes\cdots\otimes f_{n-1})(\pi(k))$, the sum $S(\gamma)$ becomes exactly the 2-dimensional path sum of $g \otimes f_n$. By the $n=2$ case, the optimal path is obtained by sorting the edge vectors of $Q_2(g)$ and $Q_2(f_n)$ by their slopes. This implies that steps orthogonal to $e_n$ can only happen when the $n$-th coordinate is equal to $i_j^{(n)}$ for some $j$. 

By symmetry, this property holds for every direction $e_l$, meaning that $\gamma$ only changes direction at the points of the subgrid $I\big([m_1]\times \dots \times [m_n]\big)$. In other words, $\gamma$ is the extension $I(\sigma)$ of some lattice path $\sigma$ from $(0,\dots,0)$ to $(m_1,\dots,m_n)$.

Now consider any two adjacent steps of $\sigma$. Since the corresponding part of $\gamma$ is 2-dimensional, by the $n=2$ case, since $\gamma$ is maximal, the slope of the corresponding edge vectors must be in non-decreasing order. Since this applies to any pair of adjacent steps, $\sigma$ is globally sorted by the slopes of the corresponding edge vectors. This uniquely identifies $\sigma = \sigma_{\mathrm{max}}$, and hence $\gamma = \gamma_{\mathrm{max}}$, proving the claimed inequality.
\end{proof}
Theorem~\ref{thm:mvformula} is immediate upon combining Lemma~\ref{lem:mixedvolumesmaller} with Lemma~\ref{lem:mixedvolumegreater}.

\section{Inequality between mixed volumes and \texorpdfstring{$p$}{p}-norms} \label{sec:ineqB}
The goal of this section is to prove Theorem \ref{thm:VolLowerbound}. 
\subsection{Exponential parameterization}
We begin by giving a more convenient parameterization of the class of convex polygons $Q_2(\I)$. Recall that for a collection of closed intervals $\I$ in $\r$ we define
	\[Q_n(\I)= \mathrm{conv}\Big(\bigcup_{I\in\I}I^n\Big). \]
Also recall that, letting $p_0,p_1,\dots,p_{m+2}$ denote the vertices of the polygon $Q_2(\I)$ in the half-plane $\{(x,y)\colon x\ge y\}$, ordered from left to right, we let the edge vectors of $Q_2(\I)$ be the sequence $(v_0,v_1,\dots,v_{m+1})$ given by $v_i=p_{i+1}-p_i$.

We now let $I_j=[a_j,b_j]$ for $j=0,1,2,\dots,m$ be a minimal sequence of intervals of $\I$, ordered by increasing right endpoint, such that $Q_2(\I)= \mathrm{conv}\Big(\bigcup_{j=0}^m I_j^2\Big)$. Let us write
	\[\frac{b_j-b_{j-1}}{a_j-a_{j-1}}=x_j\,,\qquad\frac{a_j-a_{j-1}}{b_{j-1}-a_{j-1}}=\frac{1-x_j^{s_j}}{1-x_j}\,,\quad j=1,\dots,m.\]
Then $|I_j|/|I_{j-1}|=x_j^{s_j}$, and now the non-trivial edge vectors of $Q_2(\I)$ are
	\[v_j=\bigg(x_j\frac{|I_{j-1}|-|I_j|}{1-x_j}, \frac{|I_{j-1}|-|I_j|}{1-x_j}\bigg) = 
	        |I_{j-1}|\bigg(x_j\frac{1-x_j^{s_j}}{1-x_j}, \frac{1-x_j^{s_j}}{1-x_j}\bigg)\,,\]
for $j=1,\dots,m$ when $s_j>0$. For $x_j=1$ all expressions should be interpreted in the limiting sense as $x_j\to1$. Hence the polygon $Q_2(\I)$ is uniquely determined by $I_0$ and the two $m$-tuples $x_1> x_2> \dots > x_m > 0$ and $s_1,\dots,s_m\ge0$. It is convenient to allow $s_j=0$ to encode slopes $x_j$ that don't appear in $Q_2(\I)$, especially since we can associate the same $x$ to different families of intervals; for convenience, we will allow the corresponding minimal sequences of intervals to have repetitions encoding the indices $j$ for which $s_j=0$. 

We will call $(x,s)$ \emph{exponential parameters} of the polygon $Q_2(\I)$. Note that if all $s_j>0$ then $x_j^{-1}$ are simply the slopes of the edge vectors of $Q_2(\I)$. A useful example to have in mind is the following: if $f\colon[n]\to\r_{\ge0}$ is a geometric progression, $f(j)=x^j$, then $Q_2(f)$ has exponential parameters $(x,n)$.

To motivate the next definition, note that if $Q_n(\I)=Q_n(f)$ for some log-concave function $f:[k]\to\r_{\geq 0}$ with $f(0)=1$, then in the volume identity \eqref{eq:volumeqf} we get $i_j=j$ and we obtain
    \[
    V(Q_n(f)) = 1+\sum_{j=1}^{k}\frac{1-x_j^{-n}}{1-x_j^{-1}}\prod_{i=1}^{j}x_i^{n}\,,
    \]
    with exponential parameters $x_j=f(j)/f(j-1)$ and $s_j=1$. For $x\in\r_{>0}^{k}$ and $s\in\r_{\ge0}^{k}$ we define
\begin{equation} \label{eq:Wdef}
	W(x;s) := 1+\sum_{j=1}^{k}\frac{1-x_j^{-s_j}}{1-x_j^{-1}}\prod_{i=1}^{j}x_i^{s_i}\,.
\end{equation}
The volume identity above then becomes $V(Q_n(f),\dots,Q_n(f))=W(x;{\bf 1} + \dots +{\bf 1})$ with ${\bf 1}=(1,\dots,1)\in \r^k$.
It turns out that this identity generalizes to the off-diagonal case of mixed volumes involving arbitrary families of intervals.
\begin{lemma} \label{lem:generalmv}
	Let $x_1>\dots>x_k>0$ and assume that $Q_2(\I^{(j)})$, $j=1,\dots,n$, has exponential parameters $(x,s^{(j)})$ as defined above. Then
	\begin{equation} \label{eq:generalmv}
		V(Q_n(\I^{(1)}),Q_n(\I^{(2)}),\dots,Q_n(\I^{(n)})) = |I_0^{(1)}|\cdots |I_0^{(n)}|\,W(x;s^{(1)}+\dots+s^{(n)})\,.
	\end{equation}
\end{lemma}
\begin{proof} The proof is similar to the proof of Lemma~\ref{lem:mixedvolumesmaller}.
	It suffices to prove the claim for $x_j\ne1$, since the general case follows by continuity. We also may assume that $s^{(1)}+\dots+s^{(n)}$ has no zero components. Let us write $Q_n(\I_t)=t_1Q_n(\I^{(1)})+t_2Q_n(\I^{(2)})+\dots+t_nQ_n(\I^{(n)})$, with a minimal set of intervals 
	$I_j=I_j(t)=t_1I_j^{(1)}+ \dots+t_n I_j^{(n)}$ for $j=0,\dots,k$, where $I_0^{(l)},\dots,I_k^{(l)}$ is an enlarged minimal set of intervals of $\I^{(l)}$, where we allow repetitions if $s_j^{(l)}=0$. In the above notation, formula ~\eqref{eq:vqn} becomes
	\[ V(Q_n(\I_t)) = |I_0|^n
	+ \sum_{j=1}^{k}\frac{|I_{j}|^n-|I_{j-1}|^n}{1-x_j^{-1}}\,.\]
	Since $|I_j|=\sum_{i=1}^{n}t_i|I_j^{(i)}|$, applying $\frac{1}{n!}\partial_{t_1}\cdots\partial_{t_n}$ gives
	\[V(Q_n(\I^{(1)}),Q_n(\I^{(2)}),\dots,Q_n(\I^{(n)})) =  \prod_{i=1}^{n}|I_0^{(i)}|
	+ \sum_{j=1}^{k}\frac{\prod_{i=1}^{n}|I_{j}^{(i)}|-\prod_{i=1}^{n}|I_{j-1}^{(i)}|}{1-x_j^{-1}}\,,\]
	from which the claim follows by recalling that $|I_{j}^{(i)}|/|I_{j-1}^{(i)}|=x_j^{s_{j}^{(i)}}$.
\end{proof}

\subsection{Proof of Theorem~\ref{thm:VolLowerbound}}
Note that Lemma~\ref{lem:generalmv} together with the Aleksandrov--Fenchel inequality~\eqref{eq:alexandrov-fenchel} implies the following.
\begin{corollary} \label{cor:logconcave}
	For fixed $x_1>\dots>x_k>0$, the function $s\mapsto W(x;s)$ is log-concave.
\end{corollary}
\begin{proof}
	It is enough to prove midpoint log-concavity. Choose interval systems whose exponential parameters are $(x,s/2)$ and $(x,t/2)$. Lemma~\ref{lem:generalmv} in dimension $2$, followed by the Aleksandrov--Fenchel inequality, gives
	\[W(x;(s+t)/2)^2 \ge W(x;s)W(x;t),\]
	proving the claim.
\end{proof}

To deal with $p$-norms we will need the following lemma. We first record a simple to derive identity for the function $W$ that will be useful below:
\begin{align}\label{eq:WfuncId}
W(x;s) = W(x^{(1)};s^{(1)})  +  x_1^{s_1} \dots x_{l_0}^{s_{l_0}}W(x^{(2)};s^{(2)}) - x_1^{s_1} \dots x_{l_0}^{s_{l_0}}
\end{align}
whenever $x=(x^{(1)},x^{(2)}) \in \r_{>0}^{l_0} \times \r_{>0}^{k-l_0}$, $s=(s^{(1)},s^{(2)}) \in \r_{\geq 0}^{l_0} \times \r_{\geq 0}^{k-l_0}$ and $l_0\in \{0,\dots,k\}$.

\begin{lemma} \label{lem:lpboundF}
	Assume that $f\colon \z_{\ge0}\to\r_{\ge0}$ is finitely supported and $f(0)>0$. Let $(x,s)$ be the exponential parameters of the polygon $Q_2(f)$. Then for any $p\ge1$ we have
	\[\|f\|_p \le f(0)W(x^p;s)^{1/p} \,,\]
	where $W$ is defined in~\eqref{eq:Wdef} and $x^p=(x_1^p,\dots,x_k^p)$. Moreover, if $f$ has support of size $1+N$, then $s_1+\dots+s_k\le N$.
\end{lemma}
\begin{proof}
	The proof is by contradiction. By homogeneity and after removing zero values, we may assume that $f(0)=1$ and $\operatorname{supp}(f)=[N]$.
	For $N=1$ both inequalities become identities. Assume now that one of the two claims fails and let $f$ be the counterexample with the minimal possible value of $N$. We first claim that no square $[\sum_{i<j}f(i),\sum_{i\le j}f(i)]^2$, $j=1,\dots,N-1$, touches the boundary of $Q_2(f)$. Indeed, if such a square touches the boundary, then, as in the proof of Lemma~\ref{lem:mixedvolumegreater}, the polygon $Q_2(f)$ decomposes into the union of the two polygons associated to the restrictions of $f$ to $\{0,\dots,j_0\}$ and $\{j_0,\dots,N\}$, with intersection equal to the square of side length
$
f(j_0)=\prod_{i=1}^{l_0} x_i^{s_i},
$
where $l_0$ is the corresponding boundary index in the exponential parameterization.
Using \eqref{eq:WfuncId} with this split of \((x,s)\), it is not hard to see that one of the two polygons will give a counterexample with smaller support.
	
	Now suppose that $N\ge3$ and consider any $1\le l<N-1$ and assume that $f(l)\le f(l+1)$ (in the other case the argument will be symmetric). Define, for $t\in[0,f(l)]$, $f_t(l)=f(l)-t$, $f_t(l+1)=f(l+1)+t$, and $f_t(j)=f(j)$ for all other $j$. Then $\|f_t\|_p$ is increasing in $t$ and for small values of $t$, we have $Q_2(f_t)=Q_2(f)$. Consider the first time $T$ for which either $f_T(l)=0$ or one of the two moving squares touches the boundary of the fixed polygon. If $f_T(l)=0$, then $f_T$ is a counterexample with smaller support. If one of the moving squares touches the boundary, then $f_T$ is a counterexample with the same support and an internal square touching the boundary, contradicting the previous paragraph. Thus there must be a counterexample with $N=2$.

	We now show no counterexample with $N=2$ exists. If $Q_2(f)$ has two non-trivial edge vectors, then both inequalities become identities. Otherwise, $Q_2(f)$ is a hexagon, and writing $x=x_1,s=s_1$, we have $f(0)=1$, $f(2)=x^{s}$, and $f(1)=\frac{x-x^{s}}{1-x}$ (this immediately implies $s\ge1$). Since $Q_2(f)$ is a hexagon, $f(1)^2\le f(2)f(0)$, so $\frac{x-x^{s}}{1-x} \le x^{s/2}$, which implies $s\le 2$. We obtain
	\[
	\|f\|_p^p = 1+x^{sp} + x^p\Big(\frac{1-x^{s-1}}{1-x}\Big)^{p} \le 1+x^{sp} + x^p\frac{1-x^{p(s-1)}}{1-x^p} = W(x^p;s),
	\]
	where the above inequality is a special case of  Lemma~\ref{lem:1Dkaramata}.
	\end{proof}

The most difficult part of our proof is the following inequality for the function $W$.
\begin{theorem}\label{thm:1Dineq}
	Let $r,p \geq 1$ and $0<\al \leq \be$ be such that
	\begin{align}\label{id:condat1}
		(1+r(s_1+\dots+s_k)/\al)^{1/r} =(1+p(s_1+\dots+s_k)/\be)^{1/p}\,.
	\end{align}
	Then for any $x_1\ge x_2\ge\dots \ge x_k> 0$
	we have
	\begin{equation}\label{eq:ineqBsymmetric}
		W(x^{\alpha};\tfrac{r}{\alpha}s)^{1/r} \ge 
		W(x^{\beta};\tfrac{p}{\beta}s)^{1/p}\,.
	\end{equation}
	Equality is attained if $x=(1,\dots,1)$.
\end{theorem}
The proof will be given in the next subsection. Assuming this inequality, we can now finish the proof of Theorem~\ref{thm:VolLowerbound}.
\begin{proof}[Proof of Theorem~\ref{thm:VolLowerbound}]
	After translating the supports and deleting zero values, we may assume that $f_j(0)>0$ for $j=1,\dots,n$; by homogeneity, we normalize further to $f_j(0)=1$.
	Let $(x,s^{(j)})$ be the exponential parameters of $Q_2(f_j)$ and set $S_j=\sum_{i}s_i^{(j)}$ and $S=\sum_j S_j$. By~\eqref{eq:generalmv} together with log-concavity of $W$ we have
	\[V(Q_n(f_1),\dots,Q_n(f_n)) = W(x;s^{(1)}+\dots+s^{(n)}) \ge \prod_{j=1}^{n}W\Big(x;\tfrac{s^{(j)}M}{m_j}\Big)^{m_j/M}\,.\]
	If $S_j=0$, the $j$-th factor is constant. Otherwise, define $q_j$ by
	\[(1+S_jM/m_j)^{m_j/M}=(1+S_j)^{1/q_j}.\]
	Using Theorem~\ref{thm:1Dineq} with $\alpha=1$, $r=M/m_j$, $\beta=p=q_j$ and then Lemma~\ref{lem:lpboundF}, we obtain that
	\[W\Big(x;\tfrac{s^{(j)}M}{m_j}\Big)^{m_j/M} \ge W(x^{q_{j}};s^{(j)})^{1/q_{j}} \ge \|f_j\|_{q_{j}}\,.\]
	Note that by Bernoulli's inequality $(1+S_jM/m_j)^{1/r} \le 1+S_j$, so $q_j\ge1$, and both Theorem~\ref{thm:1Dineq} and Lemma~\ref{lem:lpboundF} can indeed be used. It remains to show that $q_j\le p_{M,m_j}$. This is equivalent to
		\[\frac{M\log(1+S_j)}{m_j\log(1+S_jM/m_j)} \le \frac{M\log(m_j+1)}{m_j\log(M+1)}\,.\]
	Since $t\mapsto \frac{\log(1+t)}{\log(1+at)}$ is monotone increasing for any $a>1$, the above inequality is equivalent to $S_j\le m_j$, which holds by Lemma~\ref{lem:lpboundF}. Therefore $\|f_j\|_{q_j}\ge \|f_j\|_{p_{M,m_j}}$, and the desired inequality follows.
\end{proof}

\subsection{Proof of Theorem~\ref{thm:1Dineq}}

Notice that if we set $\la=r s_1$ and $\ka=p s_1$, then the $k=1$ case of Theorem \ref{thm:1Dineq} becomes
\begin{align*}
\bigg[\frac{X^{\la+\al}-1}{X^{\al}-1}\bigg]^{1/\la} \geq \bigg[\frac{X^{\ka+\be}-1}{X^{\be}-1}\bigg]^{1/\ka} \quad \text{for} \ \ X\geq 0, \ (1+\la/\al)^{1/\la} = (1+\ka/\be)^{1/\ka} \ \text{ and } \ \al\leq \be.
\end{align*}
This result can be deduced from a more general result of P\'ales \cite[Thm. 1]{Pal}, however we present a simpler proof in Lemma \ref{lem:1Dkaramata} in the Appendix via Karamata's inequality. Both conditions on the parameters are indeed necessary, as can be seen by considering $X=1$ and $X\to 0^+$.

Consider now $\del_j:=s_j(s_1+\dots+s_k)^{-1}$. By compactness we can assume that $s_1+\dots+s_k>0$. We then define the concave function $u:[0,1]\to \r$ by
$$
e^{u(t)}=x_1^{\del_1}\cdots x_{j-1}^{\del_{j-1}}x_j^{t-\sum_{i=1}^{j-1}\del_i} \quad \text{if} \quad \sum_{i=1}^{j-1}\del_i\leq t< \sum_{i=1}^{j}\del_i.
$$
We obtain that
\begin{align*}
e^{\la u(0^+)}+\la \int_0^1 \frac{u'(t)}{1-e^{-\al u'(t)}}e^{\la u(t)} dt  = 1 + \sum_{j=1}^k x_1^{r s_1}\cdots x_{j-1}^{r s_{j-1}} \frac{x_j^{r s_j}-1}{1-x_j^{-\al}} = W(x^{\alpha};rs/\alpha) .
\end{align*}
Finally, a routine approximation of concave functions by piecewise linear concave functions shows that Theorem \ref{thm:1Dineq} is equivalent to the following Theorem \ref{thm:continuous1Dineq}  for $\la=rS$ and $\ka=pS$, where $S=s_1+\dots+s_k$.

\begin{theorem}\label{thm:continuous1Dineq}
Let $\la,\al,\ka,\be>0$ be parameters such that $ \al \leq \be$ and
\begin{align}\label{id:continuouscondat1}
(1+\la/\al)^{1/\la} = (1+\ka/\be)^{1/\ka}\,.
\end{align}
Let $u:[0,1]\to \r$ be concave.  Then
$$
\bigg[e^{\la u(0^+)}+\la \int_0^1 \frac{u'(x)}{1-e^{-\al u'(x)}}e^{\la u(x)} dx\bigg]^{1/\la} \geq \bigg[e^{\ka u(0^+)}+\ka \int_0^1 \frac{u'(x)}{1-e^{-\be u'(x)}}e^{\ka u(x)} dx\bigg]^{1/\ka}.
$$
Equality holds if $u$ is constant. 
\end{theorem}

This result may be of independent interest in convex analysis, since it closely resembles, in form, the sharp inequalities for log-concave distributions studied in~\cite{FG,MNR,MNT} and the references therein.

\begin{proof}[Proof of Theorem \ref{thm:continuous1Dineq}] {\bf Step 1. }(Unfolding).
By scaling, we may assume that $u(0^+)=0$. Moreover, by a routine pointwise approximation procedure, we may assume that $u\in C^2([0,1])$ and is strictly concave. First we make a flow argument to reduce the number of parameters. Define the functions
$$
B(x) = \frac{x}{e^x-1}, \quad G(x) = (x+1)\log(x+1) \quad \text{and} \quad H(x)=\frac{G(x)-x}{x}.
$$
Let $c>0$ be such that $(1+\la/\al)^{1/\la}= e^c$. We then see that the points $(\ka,\be)$ and $(\la,\al)$ belong to the same curve given in parametric form by
$$
s\in \r_{\ge0} \mapsto \big( s, c^{-1}B(cs)\big).
$$
We conclude that Theorem \ref{thm:continuous1Dineq} is equivalent to
	$$
	\frac{\d}{ds}\left( {s^{-1}\log \psi_{s,c^{-1}B(cs)}(u)} \right)  \geq 0.
	$$
Note that $\psi_{\la,\al}(lu)=\psi_{l\la,l\al}(u)$ for any $l>0$, so by replacing $u$ by $cu$ we may assume that $c=1$. Computing the derivative directly and then multiplying by $s \psi_{s,B(s)}(u)$ produces the expression
\begin{align}\label{eq:uform}
{-s^2B'(s)} \int_0^1 \frac{u'(t)^2}{(1-e^{-B(s)u'(t)})^2}{e^{s u(t)-B(s)u'(t)}}dt  & + s\int_0^1 \frac{(su(t)+1)u'(t)}{1-e^{-B(s)u'(t)}}e^{s u(t)}dt \\
& - G\left(s\int_0^1 \frac{u'(t)}{1-e^{-B(s)u'(t)}}e^{s u(t)}dt\right).
\end{align}
We now apply the following Legendre transformation
	\[w(z) =(e^s-1)\sup_{0\leq t\leq 1} ( B(s) u(t)- z t).\]
Let $a=B(s)u'(1)$ and $b=B(s)u'(0)$.  In this situation, for any $a<z<b$, the maximum is attained when $u'(t)=z/B(s)$ and  we obtain
\begin{align}\label{eq:legendreids}
w(z)=s u \circ u'^{(-1)}(z/B(s)) - (e^s-1) z u'^{(-1)}(z/B(s)).
\end{align}
In particular, $w'(z)=-(e^s-1) u'^{(-1)}(z/B(s))$ and $s u \circ u'^{(-1)}(z/B(s)) = w(z)-zw'(z)$.
Thus $w\in C^2([a,b])$, $w$ nonincreasing and $w(b)=w'(b)=0$. Observe also that $\frac{-sB'(s)}{B(s)} = H(e^s-1)=H\left(\int_a^bw''(z) dz\right)$. After the change of variables $t = u'^{(-1)}(z/B(s))$, we conclude that  \eqref{eq:uform} is equal to
\begin{align}\label{eq:wform}
H\left( \int_a^b w''(z)dz \right) \int_a^b \frac{z^2e^{-z+w_0(z)}}{(1-e^{-z})^2}w''(z)dz  & +  \int_a^b \frac{z(w_0(z)+1)e^{w_0(z)}}{1-e^{-z}}w''(z)dz \\
& - G\left( \int_a^b \frac{ze^{w_0(z)}}{1-e^{-z}}w''(z)dz\right),
\end{align}
where we will always set $w_0(z):=w(z)-zw'(z)$. It is now enough to show that \eqref{eq:wform} is nonnegative for any convex function $w \in C^2([a,b])$ such that $w(b)=w'(b)=0$.

{\bf Step 2. }(Monotonicity). Let $U_{a,b}(w)$ be the quantity in \eqref{eq:wform}.  It is enough to show that $\partial_a U_{a,b}(w) \leq 0$, since clearly $U_{b,b}(w)=0$. Setting $\si=\int_a^b w''(z) dz$ and using that $G'(x)=1+\log(1+x)$ we obtain
\begin{align*}
-\frac{e^{-w_0(a)}(1-e^{-a})}{a}\partial_a U_{a,b}(w)  & = \frac{e^{-w_0(a)}(1-e^{-a})H'(\si)}{a}\int_{a}^b \frac{z^2{e^{w_0(z)}}w''(z)}{(1-e^{-z})(e^z-1)}dz
\\  & \quad +H(\si) \frac{a}{e^{a}-1}  +    w_0(a)-\log\bigg(1+ \int_a^b \frac{ze^{w_0(z)}w''(z)}{1-e^{-z}}\bigg) \\
& \geq H'(\si)e^{-w_0(a)}\int_{a}^b \frac{z{e^{w_0(z)}}w''(z)}{e^z-1}dz +H(\si) \frac{a}{e^{a}-1}   \\  & \quad - \log\bigg(1+ e^{-w_0(a)}\int_a^b \frac{ze^{w_0(z)}w''(z)}{e^{z}-1}dz\bigg)\\
& =  H'(\si)\tau +H(\si) \frac{a}{e^{a}-1}   - \log (1+\tau),
\end{align*}
where we set $\tau =e^{-w_0(a)}\int_a^b \frac{ze^{w_0(z)}w''(z)}{e^{z}-1}dz$. In the first inequality above we used that $z\mapsto \frac{z}{1-e^{-z}}$ is increasing on $\r$, hence $ \frac{z}{1-e^{-z}} \geq \frac{a}{1-e^{-a}} $, and that 
\[
1+\tau =1+ e^{-w_0(a)}\int_a^b \frac{ze^{w_0(z)}w''(z)}{e^{z}-1}dz = e^{-w_0(a)}\left(1+ \int_a^b \frac{ze^{w_0(z)}w''(z)}{1-e^{-z}}\right),
\]
because $w_0'(z)=-zw''(z)$ and $\frac{1}{1-e^{-z}} = \frac{1}{e^z-1}+1$.

We claim that $\tau \leq \frac{1-e^{-\si a}}{e^{a}-1} $. Indeed, it is enough to show that $h(a)=e^{w_0(a)}\frac{1-e^{-\si a}}{e^{a}-1}-\int_a^b \frac{ze^{w_0(z)}w''(z)}{e^{z}-1}dz$ is decreasing as a function of $a$, as $h(b)\geq 0$. Recalling that $\si$ is also a function of $a$, direct computation shows that
\begin{align*}
	h'(a)e^{\si a-w_0(a)} = \frac{\si - \frac{e^{\si a}-1}{1-e^{-a}}}{e^a-1} \leq 0.
\end{align*}
This shows the desired claim.   Consider now the function 
\[
\eta(\tau,a,\si) =  H'(\si)\tau +H(\si) \frac{a}{e^{a}-1}   - \log (1+\tau)
\]
for $\tau,\si> 0$ and $a\in \r$. It suffices to show that $\eta(\tau,a,\si)  \geq 0$ when $\tau \leq \frac{1-e^{-\si a}}{e^{a}-1} $. To that end,
since $\partial_\tau \eta(\tau,a,\si) = H'(\si)-\frac{1}{1+\tau}$, we see that $\tau\in \r_{\geq 0} \mapsto \eta(\tau,a,\si)$ is decreasing for $0\leq \tau \leq \frac{1}{H'(\si)}-1$ and increasing for $\tau\geq \frac{1}{H'(\si)}-1$. If $\frac{1-e^{-\si a}}{e^{a}-1} \leq \frac{1}{H'(\si)}-1$ we have the lower bound
\[
\eta(\tau,a,\si) \geq \eta\bigg(\frac{1-e^{-\si a}}{e^{a}-1},a,\si\bigg).
\]
If otherwise $\frac{1}{H'(\si)}-1 <\frac{1-e^{-\si a}}{e^{a}-1}$, we select $a'>a$ such that $\frac{1-e^{-\si a'}}{e^{a'}-1}=\frac{1}{H'(\si)}-1$. This can be done because $a\in \r \mapsto \frac{1-e^{-\si a}}{e^{a}-1} $ is decreasing and vanishes when $a\to+\infty$. Again we obtain
\begin{align*}
	\eta(\tau,a,\si) & > H'(\si)\tau +H(\si) \frac{a'}{e^{a'}-1}   - \log (1+\tau) \\
	& \geq \bigg[H'(\si)\tau +H(\si) \frac{a'}{e^{a'}-1}   - \log (1+\tau)\bigg]_{\tau=\frac{1}{H'(\si)}-1} \\
	& = \eta\bigg(\frac{1-e^{-\si a'}}{e^{a'}-1},a',\si\bigg).
\end{align*}
In any case, it suffices to show that
\begin{equation} \label{eq:finallemma}
	\eta\bigg(\frac{1-e^{-y}}{e^{x}-1},x,y/x\bigg) = H'(y/x) \frac{1-e^{-y}}{e^{x}-1} + H(y/x)\frac{x}{e^{x}-1} - \log\Big(\frac{1-e^{-x-y}}{1-e^{-x}}\Big) \geq 0
\end{equation}
for any $x,y\in \r$ with $xy>0$. 
Note that by \eqref{eq:legendreids} we have
$$
\tau = e^{-w_0(a)}\int_a^b \frac{ze^{w_0(z)}w''(z)}{e^{z}-1}dz = e^{-su(1)} \int_0^1 \frac{u'(x)}{e^{B(s)u'(x)}-1}e^{s u(x)}dx,
$$
hence the quantity on the right hand side of \eqref{eq:finallemma} is what one would obtain for linear $u(t)$.

{\bf Step 3. }(One slope).
We now proceed to prove inequality \eqref{eq:finallemma}.
	Note that for $c,d\in \r$ with $cd>0$ we have
	\[\int_{0}^{1}\frac{at+b}{ct+d}dt = \frac{a}{c}-\frac{ad-bc}{c^2}\log\Big(\frac{c+d}{d}\Big)\,.\]
	Then
	\[H(y/x)=\int_{0}^{1}\frac{y-yt}{x+yt}dt = \int_{x}^{x+y}\Big(1-\frac{u-x}{y}\Big)\frac{du}{u}\]
	and 
	\[H'(y/x) = \frac{y/x-\log(1+y/x)}{(y/x)^2} = \int_{0}^{1}\frac{xtdt}{x+yt} = \int_{x}^{x+y}\frac{x(u-x)du}{y^2u}\,.\]
	Similarly, we have
	\[\log\Big(\frac{1-e^{-x-y}}{1-e^{-x}}\Big) = \int_{x}^{x+y}\frac{du}{e^u-1}\,.\]
	Collecting these identities together, we see that ~\eqref{eq:finallemma} equals
	\[\int_{x}^{x+y}\Big[\frac{x(u-x)}{y^2}\frac{1-e^{-y}}{e^x-1} + \Big(1-\frac{u-x}{y}\Big)\frac{x}{e^x-1}- \frac{u}{e^u-1}\Big]\frac{du}{u}\,.\]
	Note that the bracketed expression is
	\[
	B(x)+(u-x)\left(1/B(-y)-1\right)\frac{B(x)}{y}-B(u)
	\]
	where $B(s)=\frac{s}{e^s-1}$. Since $B$ is a strictly convex function, the desired inequality will follow if we can show that $c_{x,y}:=\left(1/B(-y)-1\right)\frac{B(x)}{y}$ satisfies
	\[
		c_{x,y} \geq \frac{B(x+y)-B(x)}{y}  \quad \text{and}\quad c_{x,y} \leq \frac{B(x+y)-B(x)}{y}
	\]
	when $x,y>0$ and $x,y<0$, respectively. Manipulating terms, both inequalities reduce to
	\begin{equation} \label{eq:bernoulli}
		B(x) \geq B(x+y)B(-y).
	\end{equation}
	Since~\eqref{eq:bernoulli} becomes an equality for $y=0$, it will suffice to show that the right-hand side of~\eqref{eq:bernoulli} is decreasing in $y> 0$ when $x>0$ is fixed, and increasing in $y<0$ when $x<0$ is fixed. It is straightforward to show that $B(x)$ is strictly log-concave, therefore, it is enough to prove that $x \frac{\partial}{\partial y}\Big[B(x+y)B(-y)\Big]_{y=0} < 0$. Direct computation yields
 \[
 x\left(B'(x)+\frac{B(x)}{2}\right) = x\frac{e^x(2-x)-x-2}{2(e^x-1)^2}\,.
 \]
	Since this is an even function, we only need to show it is negative for $x>0$. After rearranging the terms it becomes $\tanh(x/2)<x/2$ for $x>0$, which follows because they agree at $x=0$ and $\frac{d}{dx} \tanh(x/2) = \tfrac1{2}(1-\tanh^2(x/2)) < \tfrac1{2}$. This finishes the proof.
\end{proof}

\section*{Acknowledgments}
We are thankful to João Ramos,  Lucas Oliveira, Dimitar Dimitrov, and Don Zagier for helpful conversations.
FG acknowledges support from the following funding agencies: The Office of Naval Research GRANT14201749 (award number N629092412126), The Serrapilheira Institute (Serra-2211-41824), FAPERJ (E-26/200.209/2023 and E-26/210.245/2024) and CNPq (309910/2023-4). DR acknowledges
funding by the European Union (ERC, FourIntExP, 101078782). 

\section*{AI use disclosure}
Generative AI was used to assist with figure preparation and to identify typos and suggest minor corrections.

\section{Appendix}\label{appendix}
We reproduce here the dimension compression argument from~\cite{Be-Iv-Kr-Ma}.

\begin{proposition}
Suppose Theorem \ref{thm:main} is true for $d=1$. Then it is also true for all functions and all dimensions $d\geq 2$.
\end{proposition}
\begin{proof}
The proof is by induction. Assume Theorem \ref{thm:main} is true for $d-1$ and let $d\geq 2$. Then $ \|f_1\bar{*}\cdots \bar{*}f_n\|_1$ is equal to
\begin{align*}
\sum_{j\in\z^d}\ \max_{j_1+\dots +j_n=j} f_1(j_1)\cdots f_n(j_n)  & =  \sum_{l\in \z} \sum_{k\in\z^{d-1}} \max_{l_1+\dots+l_n=l} \max_{k_1+\dots +k_n=k} f_1(k_1,l_1)\cdots f_n(k_n,l_n) \\
& \geq \sum_{l\in \z} \max_{l_1+\dots+l_n=l}  \sum_{k\in\z^{d-1}} \max_{k_1+\dots+ k_n=k} f_1(k_1,l_1)\cdots f_n(k_n,l_n) \\
& \geq \sum_{l\in \z} \max_{l_1+\dots+l_n=l}  \|f_1(\cdot,l_1)\|_{p_{M,m_1}} \cdots \|f_n(\cdot,l_n)\|_{p_{M,m_n}} \\
& \geq  \|f_1\|_{p_{M,m_1}} \cdots \|f_n\|_{p_{M,m_n}},
 \end{align*}
where the first inequality is trivial, in the second inequality we applied the induction hypothesis, and in the last inequality we applied the case $d=1$.
\end{proof}

\begin{lemma}\label{lem:1Dkaramata}
Let $\la,\al,\ka,\be>0$ be reals such that $(1+\la/\al)^{1/\la} \ge (1+\ka/\be)^{1/\ka}$ and  $ \al \leq \be$. Then
\begin{align*}
\bigg[\frac{X^{\la+\al}-1}{X^{\al}-1}\bigg]^{1/\la} \geq \bigg[\frac{X^{\ka+\be}-1}{X^{\be}-1}\bigg]^{1/\ka}
\end{align*}
for all $X\geq 0$.
\end{lemma}
\begin{proof}
By replacing $X$ with $1/X$ it is enough to consider $X>1$. By continuity we may assume that $\al<\be$. Rescaling $(\la,\al,\ka,\be)$ leaves the conditions invariant, so we may assume that $X=e^{2}$. Then
$$
\frac{1}{\la}\log\bigg[\frac{X^{\la+\al}-1}{X^{\al}-1}\bigg] = 1+\frac{1}{\la} \int_{\al}^{(\al+\la)} \coth(u) du
$$
and
$$ \frac{1}{\ka }\log\bigg[\frac{X^{\ka+\be}-1}{X^{\be}-1}\bigg] = 1+\frac{1}{\la}\int_{\al}^{(\al+\la)} \coth(\ell(u)) du,
$$
where $\ell(u)=\ka (u-\al)/\la + \be $. Since the function $\coth(1/u)$ is increasing and convex for $u>0$, it suffices to show that
$$
\frac1{\la}\log(1+x\la/\al)=\frac1\la \int_\al^{x\la+\al} \frac{1}{u}d u \geq \frac1{\ka}\log(1+x\ka/\be)=\frac1\la \int_\al^{x\la+\al} \frac{1}{\ell(u)}d u
$$
for $0 \leq x \leq 1$, because then the lemma will follow from weak majorization (or weak Karamata's inequality). To prove the last claim, notice that both functions $1/u$ and $1/\ell(u)$ are decreasing for $\al\leq u \leq \al+\la$, and that the function 
$f(x)=\frac1{\la}\log(1+x\la/\al)-\frac1{\ka}\log(1+x\ka/\be)$ has exactly one critical point at $x=\frac{\be-\al}{\la-\ka}$. Since $f(0)=0$, $f(1)\geq 0$ and $f'(0)=1/\al-1/\be>0$, we conclude that $f(x)\geq 0$ for $0\leq x\leq 1$.
\end{proof}

\end{document}